\documentclass[12pt]{amsart}

\usepackage{amsfonts}
\usepackage{color}
\usepackage{amscd}
\usepackage{epsfig}
\usepackage{graphicx}
\usepackage{amsmath,amssymb}
\usepackage{amsthm}
\usepackage{url}
\usepackage{enumerate}

\newtheorem{thm}{Theorem}[section]
\newtheorem{lem}[thm]{Lemma}

\newtheorem{prop}[thm]{Proposition}
\newtheorem{cor}[thm]{Corollary}
\newtheorem{defn}[thm]{Definition}

\newtheorem{rmk}[thm]{Remark}
\newtheorem{rmks}[thm]{Remarks}

\def\O{{\mathcal O}}

\def\P{{\mathbb P}}
\def\I{{\mathcal I}}

\def\fm{\mathfrak m}

\def\C{{\mathbb C}}

\def\Pic{\mathop{\rm Pic}}

\def\Cl{\mathop{\rm Cl}}

\def\codim{\mathop {\rm codim}}

\def\supp{\mathop{\rm Supp}}

\def\sing{\mathop{\rm Sing}}

\title{Grothendieck-Lefschetz theorem with base locus}

\author{John Brevik}

\address{California State University at Long Beach, 
Department of Mathematics and Statistics, Long Beach, CA 90840}

\email{jbrevik@csulb.edu}

\author{Scott Nollet}

\address{Texas Christian University, Department of Mathematics, 
Fort Worth, TX 76129}

\email{s.nollet@tcu.edu}

\subjclass[2000]{Primary: 14B07, 14H10, 14H50}

\begin{document}
\bibliographystyle{plain}

\begin{abstract}
We compute the divisor 
class group of the 
general hypersurface $Y$ of a complex projective normal variety $X$ of dimension 
at least four containing a fixed base locus $Z$. 
We deduce that completions of normal local complete intersection 
domains of finite type over $\mathbb C$ of dimension $\ge 4$ are completions of UFDs 
of finite type over $\mathbb C$. 
\end{abstract}

\maketitle

\section{Introduction}\label{intro}

In his 1882 treatise on space curves \cite{noether}, Noether tacitly stated that the 
general surface $S \subset \mathbb P_{\mathbb C}^3$ of degree $d \geq 4$ has Picard 
group $\Pic S \cong \mathbb Z$, generated by $\O_S (1)$. One of the high points in 
algebraic geometry and Hodge theory occurred in the mid 1920s when Lefschetz used 
topological methods to prove Noether's statement \cite{lefschetz}. 
By the mid-1960s, algebraic geometers began to refine the statement and proof of this 
result, leading to the modern development of Noether-Lefschetz theory. The general 
problem is to determine for what spaces $X$ with ample line bundle $L$ is it true 
that the restriction map $\Pic X \to \Pic Y$ is an isomorphism for the general 
member $Y \in |H^0(X,L)|$. A historical account with numerous references can be 
found in our survey \cite{BN3}. 

Ravindra and Srinivas proved that for base-point free ample line bundles $L$ on 
a normal complex projective variety $X$ of dimension at least four, the restriction 
map $\Cl X \to \Cl Y$ is an isomorphism for general $Y \in |L|$ \cite{rs1}. 
Around this time we extended work of Lopez \cite{lopez}, proving that for any closed 
subscheme $Z \subset \mathbb P^N$ of codimension $\geq 2$ lying on a normal hypersurface, 
the class group of the general high degree hypersurface $H$ containing $Z$ is freely generated 
by $\O_H (1)$ and the supports of the codimension two irreducible components of $Z$ \cite{BN1}. 
Here we present a generalization of these results: 

\begin{thm}\label{two}
Let $X \subset \mathbb P_{\mathbb C}^N$ be a normal projective variety of dimension $\geq 3$, and let
$Z \subset X$ be a closed subscheme of codimension $\geq 2$ with  
$\I_Z (d-1)$ generated by sections. Assume  
that the 
codimension two irreducible components $Z_1, Z_2, \dots, Z_s \subset Z$ satisfy
\begin{enumerate}[(a)]
\item $Z_i \not \subset \sing X$. 
\item $Z_i$ has generic embedding dimension at most $\dim X - 1$. 
\end{enumerate}
Then the general member $Y \in |H^0(X,\I_Z (d))|$ is normal; furthermore the homomorphism 
$\displaystyle \alpha: \bigoplus_{i=1}^s \mathbb Z \oplus \Cl X \to \Cl Y$ given by 
$\displaystyle \alpha (a_1, a_2, \dots, a_s, L) = \sum_{i=1}^s a_i \supp Z_i + L|_Y$ is  
\begin{enumerate}[(i)]
\item an isomorphism if $\dim X \geq 4$. 
\item injective with finitely generated cokernel if $\dim X = 3$. 
\end{enumerate}

\end{thm}  

\begin{rmk}\label{hypoth}{\em
Hypotheses (a) and (b) in Theorem \ref{two} are the weakest 
allowing the general $Y \in |H^0(\I_Z (d))|$ to be normal, for 
if either condition fails on a codimension two component $Z_i$, 
then $Y$ is singular along $Z_i$, hence not regular in codimension one.
\em}\end{rmk}

\begin{rmk}\label{zero}{\em
With the hypotheses of Theorem \ref{two}, we note some special cases. 
\begin{enumerate}
\item[(a)] If $Z$ is empty, the restriction map induces $\Cl X \cong \Cl Y$. 
This is the theorem of Ravindra and Srinivas \cite{rs1}, but they only 
assume $\O_X (d)$ to be ample and base-point free. 
\item[(b)] More generally, $\Cl X \cong \Cl Y$ if $\codim(Z,X)> 2$. 
\item[(c)] Taking $X = \mathbb P^N$ with $N>3$ improves our earlier result \cite[Theorem 1.7]{BN1}, where 
we showed that $\Cl Y$ is generated as above, but noted neither the splitting nor the freeness of the geometric generators 
$\supp Z_i$ given here. 
\end{enumerate}
\em}\end{rmk}

\begin{rmk}\label{dim3}{\em
When $\dim X = 3$, we expect $\alpha$ to be an isomorphism for the {\it very general} member 
$Y \in |H^0(X,\I_Z (d))|$ under stronger hypothesis, such as a stronger ampleness condition on 
$\O_X (1)$. Ravindra and Srinivas were successful with the assumption that $K_X (1)$ be generated 
by global sections (for $Z$ empty) \cite{rs2}. We hope to prove this at some point in the future. 
See Section~\ref{threefolds} for further discussion.
\em}\end{rmk}

An easy induction leads to the following corollary. 

\begin{cor}\label{one}
Let $X \subset \mathbb P^N$ be a normal projective variety of dimension $\geq 4$ 
containing a closed subscheme $Z$ of dimension $r$ with $2 \leq r < \dim X -1$ and 
suppose that the irreducible components $Z_i$ of dimension $r$ satisfy 
\begin{enumerate}
\item[(a)] $Z_i \not\subset \sing X$. 
\item[(b)] $Z_i$ has generic embedding dimension at most $r+1.$
\end{enumerate}
Then the general complete intersection $Y \subset X$ of dimension $r+1$ containing $Z$ 
is normal with class group $\Cl Y$ freely generated by $\Cl X$ and the supports of the $Z_i$. 
In particular, $\Cl Y \cong \Cl X$ if $Z$ contains no components of dimension $r$. 
\end{cor}

For our applications to local rings, the special case $X = \mathbb P^N$ is most important: 

\begin{cor}\label{original}
Let $Z \subset \mathbb P^N$ be an $r$-dimensional closed subscheme and consider the 
general complete intersection $Y \subset \mathbb P^N$ of dimension $n \geq 3$ containing $Z$. 
Assume that the locus of points where $Z$ is of embedding dimension $n$ has dimension at most $n-2$. 
Then $Y$ is normal and 
\begin{enumerate}
\item[(a)] If $n>r+1$, then $\Cl Y \cong \mathbb Z$ is generated by $\O_Y (1)$. 
\item[(b)] If $n=r+1$, then $\Cl Y$ is freely generated by $\O_Y (1)$ and the supports of the $r$-dimensional irreducible components of $Z$. 
\end{enumerate}
\end{cor}

The first few sections are devoted to proving Theorem \ref{two}. 
In section four we give an application to local commutative algebra, 
namely that local complete intersection normal domains of dimension at least 
three are completions of UFDs.  
In the last section we note some special cases where Theorem \ref{two} 
holds for $\dim X = 3$ and also examples where our proof fails in this case. 
We work over an algebraically closed field of characteristic zero, but this 
hypothesis is only used in Proposition \ref{normaladjustment} when we invoke the 
theorem of Ravindra and Srinivas \cite{rs1}. 

\section{Proof of the main theorem}

As in Theorem \ref{two}, let $X \subset \mathbb P^N$ 
be a normal projective variety of dimension at least three containing a 
closed subcheme $Z \subset X$ of codimension $\geq 2$ and assume that 
the codimension two irreducible components $Z_1, Z_2, \dots, Z_s$ ($s=0$ is possible) satisfy
\begin{enumerate}
\item[(a)] $Z_i \not \subset \sing X$ and 
\item[(b)] The generic embedding dimension of $Z_i$ is at most $\dim X - 1$. 
\end{enumerate}

The hypersurfaces containing $Z$ define a rational map $X \to \mathbb P H^0(X, \I_Z (d))^*$ with scheme-theoretic base locus $Z$ for $d$ large. The blow-up 
$f: \widetilde X \to X$ at $Z$ removes the points of indeterminacy 
of this rational map \cite[II, Ex. 7.17.3]{AG}, giving an honest morphism 
$\sigma: \widetilde X \to \mathbb P = \mathbb P H^0(X, \I_Z (d))^*$. If $\I_Z (k)$ is generated by global sections, 
then $\sigma$ is a closed immersion for $d>k$ \cite[Prop. 4.1]{PS} 
and we obtain the diagram 
\begin{equation}\label{blowup}
\begin{array}{ccccc}
E & \subset & \widetilde X & \stackrel{\sigma}{\hookrightarrow} & \mathbb P H^0(X, \I_Z (d))^* \\
\downarrow & & f \downarrow & & \\
Z & \subset & X & & 
\end{array}
\end{equation}
in which $\sigma$ is a closed immersion and 
$\sigma^*(\O_{\mathbb P} (1)) = f^* (\O_X (d)) \otimes \O_{\widetilde X} (-E)$, $E$ being the exceptional divisor. 

\subsection{Blowing up the base locus}

To have some control over the blow-up $\widetilde X$, we draw attention 
to the bad locus $B \subset Z$, over which we have little control, defined 
as follows:  

\begin{defn}{\em The {\it bad locus} $B \subset Z$ consists of the points 
$z \in Z$ such that 
\begin{enumerate}
\item[(a)] $X$ is singular at $z$ or 
\item[(b)] $\I_Z$ is not $2$-generated at $z$ or 
\item[(c)] $Z$ has embedding dimension $\geq \dim X$ at $z$. 
\end{enumerate}
\em}\end{defn}

\begin{rmk}\label{codimB}{\em
Conditions (a) and (b) in the hypothesis of Theorem \ref{two} imply that 
$\codim(B,X) \geq 3$. Indeed, $I_Z$ is not 2-generated along irreducible 
components of codimension $\geq 3$ and for an irreducible component 
$Z_i \subset Z$ of codimension two, $X$ is generically smooth along $Z_i$ and 
$Z$ has generic embedding dimension $\leq \dim X -1$.
\em}\end{rmk}

The blow-up $f:\widetilde X \to X$ along $Z$ is quite 
nicely behaved away from $f^{-1}(B)$:  

\begin{prop}\label{properties}
Let $f: \widetilde X \to X$ be the blow-up as in diagram (\ref{blowup}). Then 
\begin{enumerate}
\item $f:E - f^{-1}(B) \to Z - B$ has the structure of a $\mathbb P^1$-bundle over $\bigcup Z_i$. 
\item For $z \in Z-B$, $\sigma$ embeds the fiber 
$f^{-1}(z) \cong \mathbb P^1$ as a straight line.
\item For $z \in Z-B$, $f^{-1}(z) \cap \sing \widetilde X$ is empty or consists of a single point.   
\item $\widetilde X - f^{-1}(B)$ is normal. 
\end{enumerate}
\end{prop}

\begin{proof}
At points in $Z - B$, $X$ is smooth and the codimension two components $Z_i \subset X$ 
have embedding dimension at most $\dim X - 1$, hence $Z_i$ is locally a 
divisor on a smooth subvariety of codimension one in $X$ and therefore locally defined by two equations. 
Therefore in the {\bf Proj} construction for blow-up, the exceptional divisor over $Z-B$ takes the form 
$\mathbb P_{(Z-B)} (\I_Z / \I_Z^2)$ of a $\mathbb P^1$-bundle, proving (1). As noted earlier, the closed immersion 
$\sigma$ is defined by the line bundle $f^*( \O_X (d)) \otimes \O_{\widetilde X} (-E)$, which restricts simply to $-E$ 
on a fiber $f^{-1}(z) \cong \mathbb P^1$, but $-E$ is the relative $\O (1)$ in the {\bf Proj} construction of blow-up, 
so the fiber embeds as a straight line, proving (2).

Now (3) follows from \cite[Theorem 2.1]{N}, but we recall the local calculation here to 
say more. If $Z$ has local ideal $I=(x,f)$ away from $B$, where $H$ has equation $x=0$, 
then $\widetilde X$ is locally constructed as {\bf Proj} of the graded ring 
$\oplus I^j \cong A[X,F]/(xF-Xf)$, where $A = \O_X$ is a regular ring. 
If $(x,f)$ extends to a regular system of parameters for the maximal ideal $\fm$, then the fiber is 
contained in the smooth locus of $\widetilde X$, but this is not the case if $f \in \fm^2$. 
By the Jacobian criterion, the only singularity in the central fiber occurs at the point $F=0$. 
At smooth points of the support of $Z$, the ideal takes the form $(x,y^m)$, where $(x,y)$ 
extends to a regular system of parameters. 

Clearly $\widetilde X - f^{-1}(B)$ is normal away from $E$, since $f:\widetilde X \to X$ 
is an isomorphism there. Looking along $E = f^{-1}(Z-B)$, we see that the singular locus 
has codimension $\geq 2$ because $Z-B$ has codimension two and there is at most one 
singular point in each fiber. 
Furthermore, as seen above, $\widetilde X$ is locally defined by the equation $xF-Xf$ in 
$X \times \mathbb P^1$, and hence satisfies Serre's condition $S_2$, therefore 
$\widetilde X - f^{-1}(B)$ is normal along $E$ as well. 
\end{proof}

\begin{rmk}\label{xym} {\em Note from the proof above that at points of $Z$ away from the bad locus $B$ where components of $Z$ are not intersecting, the ideal of $Z$ has the form $(x,y^m)$ in local coordinates, where $m$ is the multiplicity of the 
component at that point.} \end{rmk}

\begin{rmk}{\em
Since we put no condition on the irreducible components of $Z$ having 
codimension $> 2$, the blow-up map $f: \widetilde X \to X$ can behave 
badly along this locus. Fortunately, these components 
are contained in $B$, and we need not confront them.
\em}\end{rmk}

\subsection{A split exact sequence}

Having seen that $U=\widetilde X - f^{-1}(B)$ is normal, we can 
investigate $\Cl U$. 
By the previous proposition, $E \cap U$ is a $\mathbb P^1$-bundle 
over $\bigcup_{i=1}^{s} Z_i$, 
hence it consists of $s$ irreducible components $E_1, \dots , E_s$. Applying \cite[II, Prop. 6.5]{AG} we obtain an exact sequence 
\begin{equation}
\bigoplus_{i=1}^s \mathbb Z \to \Cl U \to \Cl (\widetilde X - E) \to 0 
\end{equation}
where the map on the left is given by images of the $E_i$. Notice that 
\[
\Cl (\widetilde X - E) \cong \Cl (X - Z) \cong \Cl X 
\]
and that via these isomorphisms the pull-back map $f^*: \Cl X \to \Cl \widetilde X$ provides a 
splitting for the surjection so that $\Cl U \cong G \oplus \Cl X$, where $G \subset \Cl U$ 
is the subgroup generated by the classes of the $W_i$. We claim that the $E_i$ have no relations, 
so that $G \cong \mathbb Z^s$ is a free abelian group: 
   
\begin{prop}\label{free}
With the notation above, there is a split-exact sequence 
\[
0 \to \bigoplus_{i=1}^s \mathbb Z \stackrel{\cdot E_i}{\to} \Cl (\widetilde X - f^{-1}(B)) \to \Cl X \to 0 
\]
where the splitting on the right is given by pull-back.  
\end{prop}  

\begin{proof}
In view of the discussion above, we need only show injectivity on the left, i.e. that the 
$W_i$ have no relations in $\Cl (\widetilde X - f^{-1}(B))$. 
To see this, intersect $X$ with a general linear subspace $L \subset \mathbb P^N$ of 
codimenion equal to $\dim X - 2$ to obtain a normal surface $S=X \cap L$ in which 
all properties of the triple $(X,Z,B)$ hold for $(X \cap L, Z \cap L,B \cap L)$ and 
in particular $S \cap B = \emptyset$. Choosing $L$ generally enough 
to meet the $Z_i$ tranversely at smooth points of their supports, the local ideal of 
$Z \cap L$ in $S$ takes the form $(x,y^m)$ as in Remark~\ref{xym} above. 
We have a commuting square 
\[
\begin{array}{ccc}
\bigoplus_{i=1}^s \mathbb Z & \to & \Cl (\widetilde X - f^{-1}(B)) \\
\downarrow & & \downarrow   \\
\bigoplus_{i=1}^t \mathbb Z & \to & \Cl \widetilde S 
\end{array}
\]
in which $t$ is the number of points in $S \cap Z$. The vertical map on the left is 
clearly injective and the horizontal map on the bottom is injective by $t$ 
applications of Lemma \ref{freepoint} which follows, so the horizontal map on the top 
is also injective. 
\end{proof}

\begin{lem}\label{freepoint}
Let $S$ be a normal projective surface and let $Z \subset S$ be a closed subscheme 
supported at a smooth point $p \in S$ which has embedding dimension at most one. 
If $\widetilde S$ is the blow-up of $S$ at $Z$, then $\Cl \widetilde S \cong \Cl S \oplus \mathbb Z$. 
\end{lem}

\begin{proof}
Since $p$ is a smooth point and $Z$ has embedding dimension equal to one, there are local 
coordinates on $S$ at which $\I_Z = (x, y^m)$ for some $m \geq 1$ and in particular is a 
local complete intersection. It follows that if $f: \widetilde S \to S$ is the blow-up at $Z$, 
then the exceptional divisor $E \cong \mathbb P^1$, but there is a unique singular point 
$q \in E \subset \widetilde S$ on the exceptional divisor \cite{N}. In any event, $E$ is an 
irreducible and reduced divisor on $S$, so we have the standard exact 
sequence \cite[II, Prop. 6.5]{AG}
\begin{equation}\label{6.5}
\mathbb Z \stackrel{E}{\to} \Cl \widetilde S \to \Cl S \to 0 
\end{equation}
and we need to verify that the left map is injective. In view of the composite map  
\[
\mathbb Z \stackrel{\cdot E}{\to} \Cl \widetilde S \to \Pic E = \mathbb Z
\]

it suffices that $\O_{\tilde S} (E)|E$ is nontrivial, but this follows from the fact that $-E$ is the canonical $\O(1)$ for the {\bf Proj} construction 
of the blow-up.
\end{proof}

\subsection{Proof of the main theorem}

We set up the proof of Theorem \ref{two} with the following, which 
applies the theorem of Ravindra and Srinivas \cite{rs1} to the blow-up above. 

\begin{prop}\label{normaladjustment}
Let $V\subseteq \P^N$ be a projective variety of dimension at least three with very ample 
divisor $\O_V (1)$ and let $\Sigma \subset V$ be a closed subset containing 
the non-normal locus of $V$. 
Then for general $H \in |H^0(\mathbb P^N, \O_V (1))|$ and $W = V \cap H$, 
$W - \Sigma$ is normal and the restriction map $\Cl (V - \Sigma) \to \Cl (W - \Sigma)$ 
\begin{enumerate}
\item is an isomorphism if $\dim V > 3$. 
\item is injective with finitely generated cokernel if $\dim V = 3$. 
\end{enumerate} 
\end{prop} 

\begin{proof}
Since $W = V \cap H$ is a Cartier divisor, it satisfies Serre's condition $S_2$ 
along $W - \Sigma$ by generality of $H$. Moreover, since $\sing W = \sing V \cap W$ in this setting, 
$W - \Sigma$ is regular in codimension one (because $V - \Sigma$ is), hence 
$W - \Sigma$ is normal. 

Now consider the normalization map $g: \widehat V \to V$. Let $B_j$ be the 
divisorial components of $\Sigma$ and let $B_{j,k} \subset \widehat V$ be those components of $g^{-1}(B_j)$ that 
map onto $B_j$. Note that the $B_{j,k}$ are all the divisorial components 
contained in $g^{-1}(\Sigma)$ because the normalization map $g$ is a finite 
morphism. Composing with the closed immersion $f: V \hookrightarrow \mathbb P^N$ 
defined by $\O_V (1)$, $f \circ g$ is also finite. We can view $g^{-1}(W)=\widehat W$ as 
a general hyperplane section of the map $\widehat V \to \mathbb P^N$. As such, 
Bertini's theorem tells us that $\widehat W \cap B_{j,k}$ give all the divisorial 
components of $g^{-1}(\Sigma) \cap \widehat W$. 

Now we apply the theorem of Ravindra and Srinivas \cite{rs1}, which states that 
the restriction map $r: \Cl \widehat V \to \Cl \widehat W$ is an isomorphism if 
$\dim \widehat V > 3$ and injective with finitely generated cokernel if 
$\dim \widehat V = 3$. 
This map takes the $B_{j,k}$ to $\widehat W \cap B_{j,k}$. 
Applying the standard exact sequence 
\cite[II, Prop. 6.5]{AG} we obtain the commutative diagram 
\begin{equation}
\begin{array}{ccc}
\bigoplus \mathbb Z  & = & \bigoplus \mathbb Z \\
\downarrow \alpha & & \downarrow \beta \\
\Cl \widehat V & \hookrightarrow & \Cl \widehat W \\ 
\downarrow & & \downarrow \\
\Cl (\widehat V - \bigcup B_{j,k}) & \to & \Cl (\widehat W - \bigcup B_{j,k}) \\
\downarrow & & \downarrow \\ 
0 & & 0  
\end{array}
\end{equation} 
in which the map $\alpha$ is given by the $B_{j,k}$ and $\beta$ is given by the 
$\widehat W \cap B_{j,k}$. Looking at the top commutative rectangle, it is evident 
that the image of $\alpha$ maps isomorphically onto the image of $\beta$ under the 
middle horizontal map. Replacing $\bigoplus \mathbb Z$ with the respective images of $\alpha$ and $\beta$ in the diagram above and applying 
the Snake Lemma we see that the bottom map is injective with a finitely generated 
cokernel that is zero when $\dim \widehat V > 3$. 
Finally, the inclusions 
\[
V - \Sigma = \widehat V - g^{-1} (\Sigma) \subset \widehat V - \bigcup B_{j,k}
\]
and 
\[
W - \Sigma = \widehat W - g^{-1} (\Sigma) \subset \widehat W - \bigcup B_{j,k}
\]
differ by closed subsets of codimension $\geq 2$ and so do not affect the class groups, 
therefore the restriction map $\Cl (V - \Sigma) \cong \Cl (W - \Sigma)$ is an 
isomorphism if $\dim V > 3$ and injective with finitely generated cokernel 
if $\dim V = 3$.  
\end{proof}

Now we prove theorem \ref{two}: 

\begin{thm} Let $X \subset \mathbb P^N$ be a normal projective variety of dimension at 
least three and $Z \subset X$ be a closed subscheme satisfying the 
conditions of Theorem \ref{two}. Then the general hypersurface 
$Y \in |H^0(\I_Z (d))|$ is normal, and the natural map 
\[
\varphi: \bigoplus_{i=1}^s \mathbb Z \oplus \Cl X \to \Cl Y
\]
generated by $\supp Z_i$ and the restriction map $\Cl X \to \Cl Y$ 
is injective. Moreover, $\varphi$ is an isomorphism if $\dim X > 3$ and has 
finitely generated cokernel if $\dim X = 3$. 
\end{thm}

\begin{proof} 
We maintain the notation of the previous section, with $B$ denoting the bad locus.
Since $Y$ is a Cartier divisor on the normal variety $X$, it satisfies Serre's 
condition $S_2$; thus to prove the first statement it suffices to show that $Y$ 
is regular in codimension one. While this can be seen directly, we will prove it 
by comparing $Y$ to its strict transform $\widetilde Y \subset \widetilde X$ in order to 
draw the remaining conclusions. The general hypersurface $H \subset \mathbb P H^0(X,\I_Z (d))^*$ intersects $\widetilde X$ in an irreducible Cartier 
divisor $\widetilde Y$ that agrees with $Y$ away from $E$, 
hence $\widetilde Y = \sigma^*(H) \cap \widetilde X$ is the strict transform of its image 
$Y \subset X$. 
Such an $H$ intersects the general 
fiber $f^{-1}(z) \cong \mathbb P^1$ transversely in a reduced point 
(see Prop. \ref{properties}), so the projection $f: \widetilde Y \to Y$ is an 
isomorphism away from $E$ and a generic isomorphism along $f^{-1}(Z-B)$. 
In particular, there is a closed subset $A \subset Z$ of codimension 
$\geq 2$ in $Y$ such that the restriction map $f:\widetilde Y - f^{-1}(A \cup B) \to Y - (A \cup B)$ 
is an isomorphism. 
Since $\widetilde X - f^{-1}(B)$ is normal, so is $\widetilde Y - f^{-1}(B)$ by 
Bertini's theorem and hence $Y - (A \cup B)$ is normal. 
In particular, $Y$ is regular in codimension one. 

To complete the argument, we apply Proposition \ref{normaladjustment} above with 
$V=\widetilde X, W=\widetilde Y$ and $\Sigma = f^{-1}(A \cup B)$ to see that 
$\Cl (\widetilde X - f^{-1}(A \cup B)) \to \Cl (\widetilde Y - f^{-1}(A \cup B))$ 
is an isomorphism for general $\widetilde Y$ if $\dim X > 3$ and an injection with finitely generated cokernel if $\dim X = 3$. Since $\codim (A,X) \geq 3$ 
and $\codim (f^{-1}(A),\widetilde X) \geq 2$ 
we have
\[
\begin{array}{ccccc}
\Cl (\widetilde X - f^{-1}(A \cup B)) & \cong & \bigoplus \mathbb Z E_i \oplus \Cl X & \cong & \bigoplus \mathbb Z Z_i \oplus \Cl X \\
\rotatebox[origin=c]{270}{$\lhook\joinrel\relbar\joinrel\rightarrow$} & & & & \rotatebox[origin=c]{270}{$\lhook\joinrel\relbar\joinrel\rightarrow$} \,\beta \\
\Cl (\widetilde Y - f^{-1}(A \cup B)) & \cong & \Cl (Y - (A \cup B)) & \cong & \Cl Y
\end{array}
\]
and the upper left horizontal isomorphism is Proposition \ref{free}. Hence $\beta$ is an isomorphism for 
$\dim X > 3$ and has finitely generated cokernel if $\dim X = 3$ by Proposition \ref{normaladjustment}.
\end{proof}

\section{Completions of UFDs} 

In~\cite[text after Thm. 1.2]{BN4}, we expressed the expectation that the completion 
of any local geometric complete intersection ring is the completion of a UFD. 
In this section, we verify this expectation for rings of dimension $\ge 3.$ The argument 
is similar to our argument for hypersurfaces, bolstered by 
(a) Corollary \ref{one} and (b) the following result, which extends 
a result of Ruiz \cite[Lemma]{ruiz}. Notationally 
${\bf x} = (x_1, x_2, \dots, x_n)$ as a vector and 
$\fm=(x_1, \dots, x_n)$ is the maximal ideal in the power series ring.  

\begin{prop}\label{analisom}
Let $f_1, \dots, f_q \in \mathbb C [[{\bf x}]]$ with $q \leq n$ 
and let $J$ be the ideal generated by the $q$-minors of the matrix 
$\left(\frac{\partial f_i}{\partial x_j}\right)$. If 
$g_1, \dots, g_q \in \mathbb C [[{\bf x}]]$ satisfy $f_i-g_i \in \fm J^2$ 
and $J \subset \fm$, 
then there is an automorphism $\phi$ of $\mathbb C [[{\bf x}]]$ with 
$\phi(f_i)=g_i$ for $1 \leq i \leq q$. 
\end{prop}

\begin{proof}
Introducing new indeterminates ${\bf y}=(y_1, \dots, y_n)$, define 
$F_i({\bf x}, {\bf y})=f_i({\bf x}+{\bf y})-g_i({\bf x})$ for 
$1 \leq i \leq q$, so that $\frac{\partial F_i}{\partial y_j}({\bf x},{\bf 0}) = \frac{\partial f_i}{\partial x_j}$ 
for all $i \leq q$ and $j \leq n$. 
According to Tougeron's implicit function theorem \cite[Prop. 1 {\em ff.}]{tougeron}, there exist $y_j({\bf x}) \in \fm J$ for $1 \leq j \leq n$ such that 
$F_i(\mathbf x, \mathbf y(\mathbf x)) = 0$ for $1 \leq i \leq q$, 
that is $f_i(\mathbf x + \mathbf y(\mathbf x)) = g_i(\mathbf x)$.

The appropriate change of variables is thus $h_i:=x_i \mapsto x_i+y_i(\mathbf x), i=1,\dots, n$. 
In fact, $\frac{\partial h_i}{\partial x_j} = \frac{\partial y_i}{\partial x_j}$ for $i\neq j$ and 
$\frac{\partial h_i}{\partial x_i} = 1 + \frac{\partial y_i}{\partial x_i}$. 
Since $y_i \in \fm J \subseteq \fm^2$, $\frac{\partial y_i}{\partial x_j}(\mathbf 0) = 0$ for all $i,j$; 
so the determinant of the transition matrix at the origin is $1$ and thus $h$ defines an automorphism of $\C[[\mathbf x]]$.
\end{proof}

\begin{rmk}{\em
The case $q=1$ is due to Ruiz \cite{ruiz} and the proof above 
emulates his. We note that if $f_1, \dots, f_q$ do not define a 
complete intersection, then $J=0$ and the result also holds trivially, 
since necessarily $J=0$ and $g_i=0$ (and $\phi$ is the identity).  
\em}\end{rmk}

Now we are in position to prove that the completion $\hat A$ 
of the local ring $A$ of a normal complete intersection variety 
is the completion of a geometric UFD, at least if 
$\dim A \geq 3$, as we predicted in our earlier paper \cite{BN3}. 
While we fully expect this result when $\dim A = 3$ as well, 
our argument below requires the statement of Corollary \ref{one} 
with $r=1$, which will be more difficult (see following section). 

\begin{thm}\label{lciufd}
Let $A = \mathbb C [[x_1, \dots, x_n]]/(f_1, \dots, f_q)$, where the 
$f_i$ are polynomials defining a variety $V$ of codimension $q$ which 
is normal at the origin. Assume $\dim A \geq 3$. Then there exists 
a complete intersection $X \subset \mathbb P^n$ of codimension $q$ 
and a point $p \in X$ such that $R = \O_{X,p}$ is a UFD and $\hat R \cong A$.  
\end{thm}

\begin{proof}
Let $f_1, \dots, f_q \in \mathbb C [x_1,x_2, \dots, x_n]$ define a variety $V$ 
which is normal at the origin, corresponding to the maximal ideal 
$\fm = (x_1, \dots, x_n)$. 
The singular locus $D$ is given by the ideal $(f_1, \dots, f_q) + J$, 
where $J$ is the ideal generated by the $q$-minors of the 
partial derivatives matrix 
$\left( \frac{\partial f_i}{\partial x_j} \right)$. 
Use primary decomposition in the ring $\mathbb C [x_1, \dots, x_n]$, 
to write 
\[
J = \bigcap_{p_i \subset \fm} q_i \cap \bigcap_{p_i \not \subset \fm} q_i = K \cap L
\]
where $q_i$ is $p_i$-primary and we have sorted into components that meet 
the origin and those that do not. Localizing at $\fm$ we find that 
$J_\fm = K_\fm$ because $L_\fm = (1)$; for example, this is clear if we 
localize the exact sequence 
\[
0 \to K \cap L \to K \oplus L \to K+L \to 0.
\]
Now if $K = (k_1, \dots, k_r)$ gives a polynomial generating set for $K$, 
the closed subscheme $Y$ defined by the ideal 
$I_Y=(f_1, \dots, f_q, k_1^3, \dots, k_r^3)$ 
is supported on the components of the singular locus of $V$ which 
contain the origin, hence has codimension $\geq q+2$ by normality of $V$ at 
the origin. The very general codimension $q$ 
complete intersection $X$ containing $Y$ satisfies $\Cl X = 0$ by Corollary \ref{one} 
and therefore $\Cl \O_{X,p}=0$ as well since the natural restriction map 
$\Cl X \to \Cl \O_{X,p}$ is surjective, hence $\O_{X,p}$ is a UFD. The local equations for $X$ have the form  
\[
g_i = \sum a_{i,j} f_j + b_{i,1} k_1^3 + \dots + b_{i,r} k_r^3
\]
for general units $a_{i,j}$ and $b_{i,j}$. In particular the matrix 
$\left( a_{i,j} \right)$ is invertible and we can multiply to obtain 
generators of the form 
\[
h_i = f_i + \sum c_{i,j} k_j^3 
\]
so that $f_i-h_i \in K^3$. 
Since $K_\fm = J_\fm$, the completions of these ideals are equal in $\mathbb C [[x_1, \dots, x_n]]$, 
so $f_i-h_i \in K^3 = J^3 \subset \fm J^2$ and  
$$\widehat \O_{X,p} = \C[[x_1, \dots, x_n]]/(g_i) = \C[[x_1, \dots, x_n]]/(h_i) \cong 
\C[[x_1, \dots, x_n]]/(f_i) = A$$ 
by Proposition \ref{analisom}.
\end{proof}

\begin{rmk}{\em
Srinivas~\cite[Questions 3.1 and 3.7]{srinivas} posed the question which ``geometric" local rings ({\em i.e.}, localizations of $\C$-algebras of finite type) are analytically isomorphic to geometric UFDs. Parameswaran and Srinivas show that the answer is affirmative for the local ring of an isolated local complete intersection singularity~\cite{parasrini}. Grothendieck~\cite[X, Cor. 3.14]{SGA} proves Samuel's conjecture that any local ring that is a complete intersection and factorial in codimension $\le 3$ is a UFD. Theorem~\ref{lciufd} settles the question for all normal local complete intersection rings save one case, namely that in which the ring has dimension $3$ and has a $1$-dimensional singularity.
\em}\end{rmk}

\section{Remarks on threefolds}\label{threefolds}

From the long history of Noether-Lefschetz theory \cite{BN3}, one expects some version of 
Theorem \ref{two} to hold when $\dim X = 3$, with a stronger hypothesis (some positivity 
condition on $K_Y$) and a weaker conclusion (the conclusion should hold for {\it very general} 
$Y$ instead of Zariski general $Y$). In this section we compare two results when there is no 
base locus and prove a positive result in the simplest case when $X$ is smooth and the base 
locus is a smooth curve. 

In the literature, the best results resembling Theorem \ref{two} for $\dim X = 3$ 
and no base locus are the following: 

\begin{thm}\label{best}
Let $X$ be a normal threefold and $f:X \to \mathbb P^N$ be a morphism given by the linear system $V \subset H^0(X, \O_X (1))$, 
where $\O_X (1)$ is ample. Then the restriction map $\Cl X \to \Cl Y$ 
is an isomorphism for very general $Y \in |V|$ under the following conditions. 
\begin{enumerate}[(a)]
\item {\em (Moishezon \cite{moishezon})} $f$ is a closed immersion, $X$ is smooth and either 
$b_2(X) = b_2 (Y)$ or $h^2(Y,\O_Y)>h^2(X,\O_X)$. 
\item {\em {(Ravindra and Srinivas \cite{rs2})}} $f_* K_X (1)$ is generated by global sections. 
\item {\em (Ravindra and Tripathi \cite{rt})} The multiplication map $$H^0(X,\O_X (1)) \otimes H^0(X, K_X (1)) \to H^0(X,K_X (2))$$ is surjective and 
$H^1({\tilde X},\Omega^2_{\tilde X} \otimes \pi^*\O_X (1))=0$ for some desingularization $\pi:\tilde X \to X$.
\end{enumerate}
\end{thm}

\begin{rmks}{\em There are several points to be made about these theorems. 

\begin{enumerate}[(a)]
\item Moishezon proves his result with the original method of Lefschetz \cite{lefschetz}, using monodromy, Hodge theory, 
Lefschetz pencils and vanishing cycles (see Voisin's book \cite{HT2}). For $X$ smooth and $f$ a closed immersion, the conditions given {\it characterize} 
the conclusion, so in this sense the result is the best possible. 

\item The result of Ravindra and Srinivas is proved by purely algebraic methods in the spirit of Grothendieck's proof of 
the Grothendieck-Lefschetz theorem \cite{SGA}, using formal completions and infinitesimal neighborhoods. The threefold 
$X$ need not have a canonical divisor, but if $g: \widetilde X \to X$ is a desingularization, then the coherent sheaf 
$K_X=g_* K_{\widetilde X}$ is independent of $g$. 

\item When $X$ is smooth and $f$ is a closed immersion, the hypothesis $f_* K_X (1)$ globally generated implies 
$h^2(Y,\O_Y) > h^2(X, \O_X)$, or by duality $h^0(Y,K_Y) > h^1(X, K_X)$. The reason is that adjunction gives 
$K_Y = K_X (1)|_Y$ and hence there is an exact sequence 
\[
0 \to K_X \to K_X (1) \to K_Y \to 0 
\]
and by Kodaira vanishing the connecting homomorphism $H^0(Y,K_Y) \to H^1(X, K_X)$ is surjective, so it's enough that 
the map $H^0(X, K_X (1)) \to H^0(Y, K_Y)$ is nonzero, but if $K_X (1)$ is globally generated then the image of this 
map generates $K_Y$ as a sheaf. The converse is not true in general, see Example \ref{counter}(a) and (b) below. 

\item The recent result (c) of Ravindra and Tripathi is very similar in both statement and spirit to Joshi's variant of the 
Noether-Lefschetz theorem \cite{joshi}, the idea of the proof being to lift the line bundle to infinitesimal neighborhoods to achieve 
the surjectivity.
\end{enumerate} 
\em}\end{rmks}

In the most familiar setting, things work out very nicely:

\begin{prop}\label{nice}
Let $X \subset \mathbb P^N$ be a smooth threefold containing the 
smooth connected curve $Z$. If $K_X (d)$ and $\I_Z (d-1)$ are generated 
by global sections, then the very general degree $d$ surface $Y$ 
containing $Z$ is smooth and $\Pic Y \cong \mathbb Z \cdot Z \oplus \Pic X$.
\end{prop}

\begin{proof}
Let $\pi: \widetilde X \to X$ be the blow-up of $X$ at $Z$ with exceptional divisor $E$. Then $\widetilde X$ is smooth, and the map 
$f: \widetilde X \rightarrow \P^N = \mathbb P H^0(\I_Z (d))$ associated to the 
linear system given by the line bundle $f^* \O_{\P^N} (1) = \pi^* \O_X (d) -  E$ is a closed 
immersion by \cite[4.1]{PS} because $\I_Z (d-1)$ is generated by its sections.
Moreover, $K_{\widetilde X} = \pi^* K_X + E$ \cite[II, Ex. 8.5 (b)]{AG}, 
so 
\[
K_{\widetilde X} \otimes f^* \O_{\P^N} (1) \cong \pi^* K_X (d)
\]
is generated by 
global sections,
and therefore $\Pic \tilde X \to \Pic \tilde Y$ is an isomorphism 
for the very general $\tilde Y \in |\pi^* \O_X (d) - E|$ by Theorem \ref{best}(b). 
Now $\Pic \tilde X \cong \mathbb Z E \oplus \Pic X$ By Proposition \ref{free} 
and the projection $f: \tilde Y \to Y=f(\tilde Y)$ is an isomorphism, 
so associating $E \cap \tilde Y \subset \tilde Y$ with $Z \subset Y$ via 
the map $f$, we obtain the result. 
\end{proof} 

\begin{rmks}{\em A few comments on this proposition: 
\begin{enumerate}[(a)]
\item The same proof works for $Z$ a disjoint union of smooth curves $Z_i$, in which case $\Pic Y$ is the direct sum of $\Pic X$ and 
summands generated by the curve components of $Z$.

\item More generally, let $X \subset \mathbb P^N$ be a normal 
threefold with closed subscheme $Z \subset X$ satisfying the 
conditions of Theorem \ref{two}. Letting $\pi: \widetilde X \to X$ be 
the blow-up at $Z$ with exceptional divisor $E$, 
further assume that 
$K_X \otimes \O_{\widetilde X} (-E) \otimes \pi^*(\O_X (d))$ is generated 
by global sections. Then the conclusion of Theorem \ref{two} holds because 
we can apply Theorem \ref{best}(b)
in our proof of Proposition 
\ref{normaladjustment} to obtain an isomorphism $\Cl \tilde X \to \Cl \tilde Y$. 
\end{enumerate}
\em}\end{rmks}

\begin{rmks}\label{counter}{\em 
The method of Proposition \ref{nice} fails if $Z \subset X$ is a point. 
\begin{enumerate}[(a)]
\item Let $\pi :X \to \mathbb P^3$ be the blow-up at a point $p$ with 
exceptional divisor $E \cong \mathbb P^2$. The invertible sheaf 
$\O_X (1) = \pi ^*(\O (d)) \otimes \O_X (-E)$ is very ample for all $d > 1$ 
and $K_X \cong \pi^*(\O (-4)) \otimes \O_X 2E)$ by \cite[II, Ex. 8.5 (b)]{AG}. 
Therefore $K_X (1) \cong \pi^*(\O (d-4) \otimes \O_X(E)$ and this sheaf is not globally generated for any $d$. 
Indeed, its restriction to $E$ is $E|_E$ is $\O_{\mathbb P^2} (-1)$ via the isomorphism $E \cong \mathbb P^2$, 
so the global sections restrict to zero along $E$. 

\item On the other hand, the (very) general member $Y \in |H^0(\O_X (1))|$ is isomorphic to the 
blow-up of a degree $d$ surface $S \subset \mathbb P^3$ at the point $p$, the exceptional 
divisor for the map $Y \to S$ being $F=E \cap Y \cong \mathbb P^1$ with self-intersection number 
$-1$ and $K_Y \cong \pi^*(K_S) \otimes F$. Therefore we have an exact sequence 
\[
0 \to \pi^* K_S \to K_Y \to K_Y|_F \to 0 
\]
Thinking of $F$ as $\mathbb P^1$, we have $K_Y \cong \O_F (-1)$ and 
$H^0(F,K_Y|_F)=H^1(F,K_Y|_F)=0$ so that $h^0(Y,K_Y)=h^0(Y,\pi^* K_S) \to \infty$ 
as $d \to \infty$. Thus $h^0(Y,K_Y) > h^2(X,\O_X)$ for $d \gg 0$ and we can apply 
Theorem \ref{best} (a).

\item By contrast, consider the 4-fold multiple point $Z$ 
defined by $\I_p^2$ on a smooth threefold $X$. The blow-ups at $P$ and at $Z$ 
are isomorphic, but now the embedding is given by $\pi^* \O(d) - 2E$ and we would have to 
check instead that $\pi^* K_X (d)$ is generated by sections for $d$ large, 
which is clear enough since $K_X (d)$ is generated by sections. 
In fact, the same holds for a multiple point $Z$ defined by any power 
$\I_p^n$ for any $n > 1$. 
\end{enumerate}
\em}\end{rmks}

\end{document}